\documentclass[reqno]{amsart}

\usepackage[english]{babel}

\usepackage{amsmath}
\usepackage{amssymb}
\usepackage{amsthm}
\usepackage{mathtools}
\usepackage{graphicx}
\usepackage[hidelinks]{hyperref}

\usepackage{csquotes}

\newtheorem{theorem}{Theorem}[section]

\newtheorem{lemma}[theorem]{Lemma}

\newtheorem{corollary}[theorem]{Corollary}
\theoremstyle{definition}

\newtheorem{definition}[theorem]{Definition}
\newtheoremstyle{customNumber}
     {}         
     {}          
     {\itshape} 
     {}          
     {\bfseries} 
     {.}         
     { }         
     {\thmname{#1}\thmnumber{ #2}\thmnote{ #3}}

\theoremstyle{customNumber}

\renewcommand{\phi}{\varphi}
\renewcommand{\rho}{\varrho}
\renewcommand{\epsilon}{\varepsilon}

\newcommand{\norm}[1]{\lVert#1\rVert}
\newcommand{\abs}[1]{\lvert#1\rvert}

\DeclareMathOperator{\CAT}{CAT}
\DeclareMathOperator{\R}{\mathbb{R}}

\DeclareMathOperator{\locc}{loc}
\DeclareMathOperator{\linn}{lin}

\mathchardef\mhyphen="2D

\numberwithin{equation}{section}

\title{On a problem of Mazur and Sternbach}
\author{Giuliano Basso}
\address{Max Planck Institute for Mathematics,
Vivatsgasse 7,
53111 Bonn,
Germany}
\email{basso@mpim-bonn.mpg.de}

\begin{document}

\begin{abstract}
We investigate Problem 155 form the “Scottish Book” due to S.~Mazur and L. Sternbach. In modern terminology they asked if every bijective, locally isometric map between two real Banach spaces is always a global isometry. Recently, an affirmative answer when the source space is separable was obtained by M. Mori, using techniques related to the Mazur-Ulam theorem and its generalizations. The main purpose of this short note is to offer a different approach to Problem 155 motivated by recent advances in metric geometry. We show that it has an affirmative answer under the additional assumption that the map considered is a local isometry. 
\end{abstract}

\maketitle

\section{Introduction}

In this short note we use recent advances in metric geometry to investigate Problem 155 from the “Scottish Book" due to Mazur and Sternbach. On November~18, 1936 they asked the following question:

\begin{displayquote}
Given are two spaces \(X\), \(Y\) of type \((B)\). \(y = U(x)\) is a one-to-one mapping of the space \(X\) onto the whole space \(Y\) with the following property: For every \(x_0 \in X\) there
exists an \(\epsilon > 0\) such that the mapping \(y = U(x)\), considered for \(x\) belonging to the sphere with the center \(x_0\) and radius \(\epsilon\), is an isometric mapping. Is the mapping \(y = U(x)\) an isometric transformation? This theorem is true if \(U^{-1}\) is continuous. This is the case, in particular, when \(Y\) has a finite number of dimensions or else the following property: If \(\norm{y_1 +y_2} =
\norm{y_1}+\norm{y_2}, y_1 \neq 0\), then \(y_2 = \lambda y_1\), \(\lambda \geq 0\).
\end{displayquote}

In modern terminology, \(X\) and \(Y\) are real Banach spaces and a sphere with center \(x_0\) and radius \(\epsilon\) corresponds to the closed ball \(B(x_0, \epsilon)=\{ x\in X : \norm{x-x_0} \leq \epsilon\}\). In the following, all Banach spaces are assumed to be over \(\R\). Let us also clarify the terminology surrounding distance-preserving mappings. We say that \(f\colon X \to Y\) is \textit{isometric} if \(\norm{f(x)-f(y)}=\norm{x-y}\) for all \(x\), \(y\in X\). Here, we use the symbol \(\norm{\, \cdot \, }\) to denote both the norm in \(X\) and \(Y\). If every \(x\in X\) has an open neighborhood \(V\) such that \(f|_{V}\) is isometric then we say that \(f\) is a \textit{locally isometric}. Moreover, a surjective isometric map is called an \textit{isometry}. Thus, Problem 155 may be reformulated as follows: 

\begin{displayquote}
Let \(U\) be a bijective map between Banach spaces. Suppose \(U\) is locally isometric. Is it true that \(U\) is an isometry?
\end{displayquote}

For a long time no progress on this question has been recorded. In Mauldin's edited version \cite{mauldin-2015} of the “Scottish Book", it is listed as unsolved and no comments are added to it. To this date, only approximately a quarter of the problems of the book are still open. Very recently, almost ninety years after the question was initially posed, a first progress has been achieved by M. Mori (see \cite{mori2023scottish}). He proved the following result.

\begin{theorem}[Theorem~3 in \cite{mori2023scottish}] Let \(X\) be a separable Banach space, \(Y\) any Banach space and \(U\colon X\to Y\) a surjection with the following property: For every \(x_0\in X\) there exists \(\epsilon_0 >0\) such that \(U\) restricted to \(B(x_0, \epsilon_{x_0})\) is isometric. Then \(U\) is an isometry.    
\end{theorem}

This settles Mazur and Sternbach's problem when \(X\) is separable. We proceed by presenting a rough sketch of Mori's proof. In \cite{mankiewicz-1972}, P. Mankiewicz proved that every isometry between balls of Banach spaces uniquely extends to an affine isometry of the whole Banach spaces. This recovers the classical Mazur--Ulam theorem. Relying on Mankiewicz's result, Mori proved a lemma stating that if \(U(B(x_0, \epsilon_{x_0}))\) has non-empty interior, then \(U|_{B(x_0, \epsilon_{x_0})}\) extends uniquely to an affine isometry \(\bar{U}\colon X \to Y\). By a straightforward application of the Baire category theorem, it follows that there exists \(x_0\in X\) satisfying the assumption of the lemma. The proof is then finished by showing that \(\bar{U}=U\) by combining the lemma with the fact that if two affine maps agree on an open set, then they are equal.  

The main goal of this note is to introduce new techniques from metric geometry to the problem. These techniques are independent of the circle of ideas surrounding the Mazur--Ulam theorem and its generalizations \cite{mankiewicz-1972, fourneau}. Our main result reads as follows.
 
\begin{theorem}\label{thm:main}
Let \(X\) and \(Y\) be Banach spaces and \(U\colon X\to Y\) a local isometry. Then \(U\) is an isometry.
\end{theorem}

Here, a map \(f\colon X \to Y\) is said to be a \textit{local isometry} if every \(x\in X\) has an open neighborhood \(V\) such that \(f|_{V}\) is an isometry onto an open subset of \(Y\). Notice that any local isometry is locally isometric, but not vice versa.  The converse follows for example if \(U\) is a locally isometric bijection with \(U^{-1}\) continuous. Thus, the following result is a direct consequence of Theorem~\ref{thm:main}. This justifies the remark of Mazur and Sternach in their original formulation of the problem.

\begin{corollary}\label{cor:main}
Let \(X\) and \(Y\) be Banach spaces and \(U\colon X\to Y\) a locally isometric bijection. If \(U^{-1}\) is continuous, then \(U\) is an isometry. 
\end{corollary}

We emphasize that Theorem~\ref{thm:main} (and also Corollary~\ref{cor:main}) follow directly from the results proved in \cite{mori2023scottish}. They could also be established (more traditionally) using standard techniques from metric geometry concerning local isometries and covering maps (see e.g. \cite[Section~3.4]{burago-2001}). However, as already mentioned above the main purpose of this note is to introduce new techniques which hopefully may have their own interesting aspects.

The proof of Theorem~\ref{thm:main} relies on two main ingredients from metric geometry, both of which have been proved in the last decade. A local-to-global result due to B. Miesch \cite{miesch-2017} generalizing the classical Cartan-Hadamard theorem from Riemannian geometry, and a uniqueness result concerning certain selections of geodesics in a Banach space due to D. Descombes and U. Lang \cite{descombes-2015}. With these two results at hand, it is only a matter of understanding the definitions to show that the map \(U\) in Theorem~\ref{thm:main} is locally affine, and thus globally affine by Lemma~\ref{lem:aux}. The formulations of both of these results use the notion of a bicombing. Before proceeding with the proof of Theorem~\ref{thm:main}, we partly motivate and review this notion in the next section. 

\subsection{Acknowledgements} I am very grateful to Michiya Mori for pointing out to me that the definition of \(\sigma^y\) in Section~\ref{sec:three} was incorrect in an earlier version of this article.

\section{What are bicombings?} Sometimes it can prove beneficial to consider Banach spaces not as usual in the linear category but to consider them in the metric category (as metric spaces) and then to recover results about the linear structure using metric notions. A well-known result in this direction is due to J. Lindenstrauss \cite[Theorem~5]{Lindenstrauss-1964}. He showed that a dual Banach space is an absolute linear retract if and only if it is an absolute metric retract. In the following, we discuss how linear segments in a Banach space may be recovered using only metric notions. We say that \(c\colon [0,1]\to X\) is a (metric) \textit{geodesic} if 
\[
\norm{c(s)-c(t)}=\abs{s-t}\cdot \norm{c(0)-c(1)}
\]
for all \(s\), \(t\in [0,1]\). Linear segments are prime examples of geodesics. However, generally there are plenty of other geodesics connecting two given points. For example, \(x=(1, 0)\) and \(y=(-1, 0)\) in \(\ell_\infty^2=(\R^2, \norm{\, \cdot \,}_\infty)\) can be connected by the linear geodesic \([x, y](t)=(1-t)x+t y\) but also by the “tent" geodesic with midpoint \(z=(0,1)\) defined by the concatenation \([z, y] \ast [x, z]\).
In fact, any point of the metric interval \(I(x, y)=\{z\in X : \norm{x-z}+\norm{z-y}=\norm{x-y} \}\) is part of a geodesic connecting \(x\) to \(y\). So there are really many geodesics connecting the two points. 

\begin{definition}
Let \(M\) be a metric space. We say that \(\sigma \colon M \times M \times [0,1] \to M\) is a bicombing if for all \(x\), \(y\in M\), \(\sigma_{xy}:=\sigma(x, y, \cdot)\) is a geodesic connecting \(x\) to \(y\).
\end{definition}

Essentially, a bicombing is a selection of a geodesic for each pair of points in a metric space. For example, if \(X\) is a Banach space then \(\sigma^{\linn}\colon X \times X \times [0,1]\to X\) defined by
\[
\sigma^{\linn}(x, y, t)=(1-t) x+ t y
\]
is a bicombing. Bicombings are a useful tool in metric fixed point theory and are employed there under various names (see \cite{itoh-1979, reich-1990, kohlenbach-2010}). In the context of geometric group theory, it has become common to study bicombings in the light of various (weak) notions of non-positive curvature. We say that a bicombing \(\sigma\) is \textit{convex} if the function
\[
t\mapsto d(\sigma_{xy}(t), \sigma_{x' y'}(t))
\]
is convex on \([0,1]\) for all \(x\), \(y\), \(x'\), \(y'\in X\). This definition is motivated by the fact that the geodesics in a \(\CAT(0)\) space, a synthetic notion of a Riemannian manifold of non-positive sectional curvature, have this property. A deep result of Descombes and Lang shows that in Banach spaces the convexity of the bicombing uniquely characterizes linear segments.

\begin{theorem}[Theorem~3.3 in \cite{descombes-2015}]\label{thm:descombes-lang}
Let \(X\) be a Banach space and \(\sigma\) a convex bicombing on \(X\). Then \(\sigma=\sigma^{\linn}\).    
\end{theorem}

We proceed by briefly explaining the different ways how this theorem may be proved. Descombes and Lang consider the metric injective hull \(E(X)\) of \(X\) which turns out to be an injective Banach space. They showed that if \(c\colon [0,1]\to X\) is a straight geodesic (meaning that \(t \mapsto \norm{x-c(t)}\) is convex for all \(x\in X\)), then \(c\) is also a straight geodesic in \(E(X)\). Using the classification of injective Banach spaces \cite{cohen-1964}, they proved by means of a direct argument that any straight curve in \(E(X)\) is a linear segment. Since any geodesic of a convex bicombing is necessarily straight, this completes the proof.  

As it turns out Theorem~\ref{thm:descombes-lang} can also be recovered by invoking a result of S.~Gähler and G. Murphy \cite{gaehler-1981} from the 80s. They studied the Doss expectation of a measure in Banach spaces. They showed that for fixed \(x\), \(y\in X\), \(t\in [0,1]\) if \(p\in X\) satisfies
\[
\norm{p-z} \leq (1-t) \norm{x-z}+ t \norm{y-z}
\]
for all \(z\in X\), then \(p=(1-t)x+ty\). Since \(p=\sigma_{xy}(t)\) has this property, Theorem~\ref{thm:descombes-lang} follows. We remark that Gähler and Murphy's result follows directly from the strong law of large numbers in Banach spaces (see \cite[Theorem~2.2.22]{molchanov-2017} for more information). Finally, yet another proof of Theorem~\ref{thm:descombes-lang} is contained in \cite[Corollary~1.3]{basso2020extending}.

\section{Proof of the main result}\label{sec:three}

Let \(U\colon X \to Y\) be as in Theorem~\ref{thm:main}. Thus, for every \(x\in X\) there exists \(\epsilon_x>0\) such that \(U\) restricted to \(B(x, \epsilon_x)\) is an isometry onto \(B(f(x), \epsilon_x)\). Clearly, \(U\) is a local homeomorphism. Thus, by \cite[Lemma~2.5]{miesch-2017} if follows that \(U\colon X \to Y\) is a covering map. Alternatively, we could also invoke \cite[Theorem~5.2]{jaramillo}. Banach spaces are contractible and thus simply-connected. Hence, it follows that \(U\) is a homeomorphism. In the following, we show that \(U\) is isometric. Let \(V_y\) denote the open ball in \(Y\) with center \(y\) and radius \(\epsilon_y:=\epsilon_{U^{-1}(y)}\). We define \(\sigma^{y}\colon V_y \times V_y \times[0,1] \to V_y\) as follows:
\[
\sigma^y(p, q, t)=U\big((1-t)U^{-1}(p)+t U^{-1}(q)\big)
\]
for all \(p\), \(q\in V_y\) and \(t\in [0,1]\). In other words, \(\sigma^y\) is the push-forward under \(U\) of the linear bicombing on the open ball in \(X\) with center \(x=U^{-1}(y)\) and radius \(\epsilon_x\). Since \(B(y, \epsilon_y)=U( B(x, \epsilon_x))\), we find that \(\sigma^y\) is a bicombing. Clearly, if \(V_{y_1} \cap V_{y_2} \neq \varnothing\), then \(\sigma^{y_1}\) and \(\sigma^{y_2}\) agree on the intersection. Moreover, each \(\sigma^y\) is consistent with taking subsegments of geodesics, that is,
\[
\sigma^y(\sigma^y_{pq}(a), \sigma^y_{pq}(b), t)=\sigma^y(p, q, (1-t)a+t b)
\]
for all \(p\), \(q\in V_y\), all \(0\leq a \leq b \leq 1\) and all \(t\in [0,1]\). This shows that the family \(\sigma^{\locc}=\{ \sigma^y \}_{y\in Y}\) is a local bicombing in the terminology of \cite{miesch-2017}. We say that a local bicombing \(\{ \sigma^y \}_{y\in Y}\) is  \textit{convex}, if each \(\sigma^y\) is a convex bicombing. As a major tool in our proof, we use the following local-to-global theorem for convex bicombings due to Miesch. See also \cite{alexander-1990, ballmann1990singular}.

\begin{theorem}[Theorem~1.1 in \cite{miesch-2017}]\label{thm:miesch}
Let \(M\) be a complete, simply-connected, length space admitting a local bicombing \(\sigma^{\locc}\). If \(\sigma^{\locc}\) is convex, then there exists a unique convex bicombing on \(M\) which is consistent with \(\sigma^{\locc}\).
\end{theorem}

For the definition of a length space we refer the interested reader to \cite{burago-2001}. For our purposes, it suffices to know that Banach spaces are length spaces. Hence, Theorem~\ref{thm:miesch} is applicable in our situation. A geodesic \(c\colon [0,1] \to Y\) is \textit{consistent} with \(\sigma^{\locc}\) if for all \(0 \leq a \leq b \leq 1\) for which there exists \(y\in Y\) such that \(c(a)\), \(c(b)\in V_y\), we have
\begin{equation}\label{eq:consistency}
c((1-t)a+tb)=\sigma^y(c(a), c(b), t) 
\end{equation}
for all \(t\in [0,1]\). By Theorem~\ref{thm:miesch}, it follows that there is a convex bicombing \(\sigma\) on \(Y\) such that
each geodesic \(\sigma(y_1, y_2, \cdot)\), for \(y_1\), \(y_2\in Y\), is consistent with \(\sigma^{\locc}\).  But, because of Theorem~\ref{thm:descombes-lang}, we know that the geodesics of \(\sigma\) are linear segments, that is,
\begin{equation}\label{eq:linear-segments}
\sigma(y_1, y_2, t)=(1-t) y_1+t y_2
\end{equation}
for all \(y_1\), \(y_2\in Y\) and \(t\in [0,1]\). Therefore, by combing \eqref{eq:consistency} with \eqref{eq:linear-segments}, we find that \(U\) is locally affine. More precisely, for every \(x_0\in X\), it follows that
\[
U\big((1-t)x_1+t x_2\big)=(1-t)U(x_1)+t U(x_2)
\]
for all \(x_1\), \(x_2\in B(x_0, \epsilon_{x_0})\). With this at hand, it is now not difficult to finish the proof. Indeed, a straightforward (but slightly tedious) proof shows that every locally affine map is globally affine. 

\begin{lemma}\label{lem:aux}
Let \(U\colon X\to Y\) be a continuous, locally affine map between Banach spaces. Then 
\begin{equation}\label{eq:what-we-want}
U\big( (1-t) x+t y\big)=(1-t) U(x)+t U(y)
\end{equation}
for all \(x\), \(y\in X\) and \(t\in [0,1]\), that is, \(U\) is (globally) affine.
\end{lemma}

\begin{proof}
Fix distinct \(x\), \(y\in X\) and abbreviate \(x_t=(1-t) x+t y\) for all \(t\in [0,1]\).
Consider the set
\[
I=\big\{ t\in [0,1] : \forall s\in [0,1]: \, U\big((1-s) x_0+s x_t\big)=(1-s) U(x_0)+s U(x_t)\big\}.
\]
Clearly, \(I\) is non-empty. In the following, we show that \(I\) is both an open and closed subset of \([0,1]\). This implies that \(I=[0,1]\), and so in particular \eqref{eq:what-we-want} holds.

Suppose \((t_i)\), with \(t_i\in I\), is a convergent sequence with limit \(t\). For each \(s\in [0,1]\), we have
\[
U\big((1-s) x_0+s x_t\big)=\lim_{i\to \infty} U\big((1-s) x_0+s x_{t_i}\big),
\]
since \(U\) is continuous. Thus, using the definition of \(I\) and the continuity of \(U\) once again, we get that \(I\) is a closed subset of \([0,1]\).

Now, let \(t\in I\). We want to show that there exists \(\epsilon>0\) such that the open ball in \([0,1]\) with center \(t\) and radius \(\epsilon\) is contained in \(I\). Since \(U\) is locally affine, this is certainly true if \(t=0\). Moreover, by a direct computation, we find that if \(0\leq t' \leq t\), then \(t'\in I\). Thus, the statement also follows when \(t=1\). We now treat the remaining case when \(t\in I \cap (0,1) \). Again, since \(U\) is locally affine, there exist \(a\), \(b\in [0,1]\) with \(a< b\)  such that \(t=(a+b)/2\) and 
\[
U((1-s) x_a+s x_b)=(1-s) U(x_a)+sU(x_b)
\]
for all \(s\in [0,1]\). We claim that 
\begin{equation}\label{eq:affine-for-b}
U((1-s) x_0+s x_b)=(1-s) U(x_0)+sU(x_b)
\end{equation}
for all \(s\in [0,1]\). Let \(u\), \(v\in [0,1]\) be such that
\[
x_t=(1-u)x_0+u x_b \quad \text{ and } \quad x_a=(1-v) x_0+v x_t.
\]
In particular, \(ub=t\) and \(rv=a\). We have
\[
U(x_t)=\tfrac{1}{2} U(x_a)+\tfrac{1}{2} U(x_b)=\tfrac{(1-v)}{2} U(x_0)+ \tfrac{v}{2} U(x_t)+\tfrac{1}{2} U(x_b)
\]
and so
\[
U(x_t)= (1-\tfrac{1}{2-v}) U(x_0)+\tfrac{1}{2-v} U(x_b).
\]
A short computation reveals that \(u=1/(2-v)\), showing that \eqref{eq:affine-for-b} holds for \(s=u\). Let now \(0\leq s\leq u\). Then
\begin{align*}
U((1-s) x_0+s x_b)&=U( (1-\tfrac{s}{u}) x_0+\tfrac{s}{u} x_t)=(1-\tfrac{s}{u}) U(x_0)+ \tfrac{s}{u} U( x_t) \\
&=(1-\tfrac{s}{u}) U(x_0)+(\tfrac{s}{u}-s) U(x_0)+s U(x_b), 
\end{align*}
where we used that \eqref{eq:affine-for-b} is valid for \(s=u\) for the third equality. Hence, \eqref{eq:affine-for-b} holds for all \(0 \leq s \leq u\). The case \(u \leq s \leq 1\) follows with an analogous computation. This proves that \(b\in I\) and so \([0, b]\subset I\). Since \(0<t<b\), an open neighborhood of \(t\) is contained in \(I\), as desired. This finishes the proof showing that \(I\) is an open subset of \([0,1]\).
\end{proof}

Thus, by Lemma~\ref{lem:aux}, it follows that \(U\) is globally affine. Without loss of generality, we may suppose that \(U(0)=0\) and so \(U\) is a linear map. By assumption, there exists \(\epsilon>0\) such that \(\norm{U(x)}=\norm{x}\) for all \(x\in B(0, \epsilon)\). Let \(x\in X\). Then by setting \(r=\norm{x}/ \epsilon\), we find that
\[
\norm{U(x)}=r \norm{ U(\tfrac{x}{r})}=r \norm{\tfrac{x}{r}}=\norm{x}.
\]
This shows that \(U\) is an isometry, as desired. We finish this section with the straightforward proof of Corollary~\ref{cor:main}.
\begin{proof}[Proof of Corollary~\ref{cor:main}]
Let \(U\colon X \to Y\) be a locally isometric bijection such that \(U^{-1}\) is continuous. We want to conclude that \(U\) is an isometry. In view of Theorem~\ref{thm:main}, it suffices to show that \(U\) is a local isometry. Fix \(x\in X\) and set \(y=U(x)\). Let \(\epsilon>0\) be such that \(U\) restricted to \(B(x, \epsilon)\) is isometric. As \(U^{-1}\) is continuous, there exists \(\delta>0\) such that \(B(y, \delta)\subset U(B(x, \epsilon))\). This implies that \(U(B(x, \delta))=B(y, \delta)\), and so \(U\) clearly is a local isometry.
\end{proof}

\let\oldbibliography\thebibliography
\renewcommand{\thebibliography}[1]{\oldbibliography{#1}
\setlength{\itemsep}{3.5pt}}
\bibliographystyle{alpha}
\bibliography{sample}

\end{document}